\documentclass[11pt]{amsart}
\usepackage{epsfig,graphics,amsmath,amssymb,amsthm}
\usepackage[utf8]{inputenc}
\usepackage{caption,comment,marginnote}
\usepackage{epstopdf}
\usepackage{float}
\usepackage{hyperref}
\usepackage{epigraph}

\renewcommand{\contentsname}{Contents\\{\footnotesize\normalfont(A table
of contents should normally not be included)}}
\newtheorem{Theorem}{Theorem}[section]

\newtheorem{Corollary}[Theorem]{Corollary}
\newtheorem{Proposition}[Theorem]{Proposition}
\newtheorem{Lemma}[Theorem]{Lemma}
\theoremstyle{definition}
\newtheorem{Example}[Theorem]{Example}
\newtheorem{Remark}[Theorem]{Remark}

\newcommand{\Mcg}{\mathrm{Mod}}
\newcommand{\PMcg}{\mathrm{PMod}}

{\catcode`\@=11 \gdef\mnote#1{\marginpar{\tiny
 \tolerance\@M\spaceskip2.6\p@ plus10\p@ minus.9\p@\rm#1}}}

\let\Bbb\mathbb

\newcommand{\be}{\begin{equation}}
\newcommand{\ee}{\end{equation}}
\let\ge\geqslant 
\let\le\leqslant 

\let\til\widetilde

\def\sm{\smallsetminus}

\def\Z{\Bbb Z}
\def\R{\Bbb R}
\def\C{\Bbb C}
\def\Q{\Bbb Q}

\def\A{\mathcal A}
\def\P{\mathcal P}

\def\CC{\mathcal C}
\def\CCm{\mathcal C^{mult}}
\def\D{\Lambda}
\def\DD{\mathbb D}
\def\DDD{\Delta}
\def\G{\Gamma}

\def\g{\gamma}
\def\sm{\smallsetminus}
\def\M{\DD_3}

\def\PSL{\operatorname{PSL}}
\def\SL{\operatorname{SL}}

\let\a=\alpha
\let\b=\beta
\def\d{\til\Lambda}
\let\s=\sigma
\let\e=\varepsilon

\begin{document}

\title[Recognising conjugacy classes of Dehn twists in $\DD_3$]
{Recognising conjugacy classes of Dehn twists in $\DD_3$ via Dynnikov coordinates}
\author{Ferihe Atalan}
\address{Department of Mathematics\\Atilim University\\
06830 Ankara\\ Turkey}
\email{ferihe.atalan@atilim.edu.tr}
\author{Sergey Finashin}
\address{Department of Mathematics\\ METU\\
06800 Ankara\\ Turkey}
\email{serge@metu.edu.tr}

\date{\today}
 \pagenumbering{arabic}

\begin{abstract} 
We describe in terms of Dynnikov coordinates
the three orbits of the pure mapping class group $\PMcg(\DD_3)$ action on the set of essential curves in a 3-punctured disc $\DD_3$.
For any essential curve $\g\subset\DD_3$ we present an efficient algorithm to untwist $\g$ into one of the 3 basic representatives of the orbits.
In turn, we give transformation formulas between the Dynnikov and torus $\Z^2$-coordinates for multicurves.
%
Besides proving minimality of the algorithm, we give an explicit
minimal length formula in terms of the  associated
even continued fractions.
\end{abstract}

\maketitle

\setlength\epigraphwidth{.4\textwidth}
\epigraph{The journey of a thousand miles begins with a single step}{Lao Tzu}

\section{Introduction}
\subsection{Motivations and the task}
The group $\PMcg(\DD_3)$ is significant as a congruence modular group related to the geometry of elliptic curves and surfaces.
Our objects of interest, essential curves in $\DD_3$, can be looked at from different viewpoints: such curves represent non-trivial Denh twists, pant decompositions of $\DD_3$, rational tangles in $B^3$ and equivalence classes of Morse functions in $\DD_3$ with the local minima at the marked points and maximum on the boundary, etc.
Essential curves are naturally characterized by the Dynnikov coordinates in $\Z^2$, and
the dynamics of the $\PMcg(\DD_3)$-action on their set can be viewed as the dynamics in $\Z^2$.

We aimed to answer the simplest questions on this dynamics:
to characterize the basic transformations of $\Z^2$ and to find its partition into orbits.
Next, we looked for an efficient (of minimal length) {\it untwisting algorithm} transforming a curve $\g$ into one of basic patterns
and in particular, to give a formula for its length in terms of Dynnikov coordinates of $\g$.
%
%
%

\subsection{The methods and results}
Our approach is based on lifting a curve $\g\in\DD_3$ to a one-holed torus $T$ via a double covering $T\to\DD_3$ branched at the marked points and relating
torus coordinates $(p,q)\in\Z^2=H_1(T)$ with 
the Dynnikov coordinates, $(a,b)\in\Z^2$.
The coordinate transformation map $(p,q)\mapsto (a,b)$ described in Theorem 
\ref{transform-thm} 
has an amazingly simple form,
$$
(p,q)\to(\frac{|p-q|-|p+q|}2,|p|-|q|),$$
which, to our surprise, we could not find in the literature.
Geometrically, it is a piecewise linear double covering 
identifying the quotient $\Z^2/(-1)$ by $(-1)$-action $(p,q)\mapsto(-p,-q)$ with $\Z^2$ (Theorem \ref{Dynnikov-properties}).
So,  one may think of it as a piecewise linear analog of the squaring map $\C\to\C$, $z\mapsto z^2$.

We study the dynamics of the basis Dehn twists $t_c,t_d\in \PMcg(\M)$, acting
in Dynnikov coordinates. This action reveals three $\PMcg(\M)$-orbits represented by the basic curves $c$, $d$, and $e$  shown in Figure \ref{threecurves}.
We describe an algorithm giving a sequence of twists $t_c^{\pm}$ and $t_d^{\pm}$ untwisting any essential curve $\g\subset \M$ 
to one of the standard curves $c$, $d$, or $e$.
As a consequence, we find 
the distance from $\g$ to $\{c,d,e\}$ in the Cayley graph associated with the action of $t_c$ and $t_d$ on the essential curves.
This shows the efficiency of our algorithm (Theorem \ref{minimal-Dynnikov}),
which is explained by
its relation to a variant of Euclid's algorithm with even partial quotients (Proposition \ref{spin-frac}), as we
 pass back from Dynnikov coordinates $(a,b)$ to the torus coordinates $(p,q)$.

As an outcome, we present an explicit description of the three orbits of $\PMcg(\DD_3)$-action on the essential curves (or the corresponding Dehn twists).

\begin{Theorem}\label{conjugacy}
Any Dehn twist $t_\g\in\PMcg(\DD_3)$ about an essential curve $\g$ in $\DD_3$
is conjugate to $t_c$, $t_d$, or $t_e$ by an element of $\PMcg(\DD_3)$. 
Namely, in terms of Dynnikov coordinates $(a,b)$ of $\g$:
\begin{itemize}\item
 $t_\g$ is conjugate to $t_e$ if $b$ is even,
\item
$t_\g$ is conjugate to $t_c$ if $b$ is odd and $(-1)^ab>0$,
\item
$t_\g$ is conjugate to $t_d$ if $b$ is odd and $(-1)^ab<0$.
\end{itemize}
\end{Theorem}

In addition, we give an {\it untwisting algorithm} expressing $x\in\PMcg(\DD_3)$ such that $xt_{\g}x^{-1}\in\{t_c,t_d,t_e\}$,
as a product of generators $t_c,t_d\in \PMcg(\DD_3)$ and prove the efficiency of this algorithm.
This means minimality of the length of such product, which we call {\it conjugation length} relating $t_\g$ to $\{t_c,t_d,t_e\}$.
We prove equality of this length to the {\it ECF-length} 
$$|\frac{p}q|_{ECF}=\frac12(|q_0|+\dots+|q_r|),$$
where $(p,q)$ are {\it torus coordinates} of $\g$ and
 $\frac{p}q$ is expressed as
an {\it even continued fraction} (ECF) $[q_0,\dots,q_r]$ or $[q_0,\dots,q_r,1]$, where $q_i\in2\Z$, $i=0,\dots,r$ (see details in Section \ref{digression}).

\begin{Theorem}\label{length}
 The conjugation length relating $t_\g$ to 
 $\{t_c,t_d,t_e\}$ is expressed in terms of
 the Dynnikov coordinates $(a,b)$ of $\g$ as
$$|\frac{p}q|_{ECF}, \text{ where }(p,q)=\begin{cases}(|a|+b,-a), \text {if }b>0,\\(a,b-|a|), \text {if }b<0.\end{cases}.$$
\end{Theorem}

\subsection{Structure of the paper}
After reviewing the standard definition and facts in Section \ref{prelim-sec}, we describe in Section \ref{dynamics} the action of $t_c$, $t_d$, $t_e$ and $t_{e'}$ on the Dynnikov plane
(Theorems \ref{tc-td-action}and \ref{e-action}). In particular, we give a partition of $\Z^2$ into broken lines, which contain
the orbits of such action.
In Section \ref{torus-disk-sec}, we discuss the transformation from the torus to Dynnikov coordinates and uncover
its geometric properties. 
Section \ref{algorithm-sec} presents an algorithm for untwisting curves in $\M$ in terms of their Dynnikov coordinates.
Section \ref{digression} contains a digression on some basic (cf., \cite{Sch}) facts about
even continued fractions and the associated variant of Euclid's algorithm.
In Section \ref{comparison-sec}, we observe that
after passing to the torus coordinates, our algorithm in Section \ref{algorithm-sec} becomes nothing but
a version of Euclid's  for finding even continued fractions.
Finally, we derive in Section \ref{main-proofs} our main theorems \ref{conjugacy} and \ref{length}.

\section{Preliminaries}\label{prelim-sec}
In this section we recall a few standard well-known generalities in order to set up notation and conventions.

\subsection{$\Mcg(\DD_3)$ and its subgroups}\label{modular-subsect}
For a surface $S$ (which may have boundary components and marked points),  we denote by $\Mcg(S)$ its {\it mapping class group}  formed by
isotopy classes of  orientation-preserving homeomorphisms $S\to S$, preserving the boundary as well as the set of marked points {\it set-wise} (rather than point-wise).
For instance, for a disc with 3 marked points $\DD_3=(\DD^2,\{p_1,p_2,p_3\})$
group
$\Mcg(\DD_3)$ is identified with the quotient $B_3/\DDD^2$ of the braid group $B_3$,  where $\DDD=\s_1\s_2\s_1=\s_2\s_1\s_2$, and $\DDD^2$
is the generator of the center. Recall also that for a torus $\hat T$ and 
a {\it torus with one hole},  $T$, we have isomorphism $\Mcg(\hat T)=\Mcg(T)=\SL_2(\Z)$, where 
 with the class of a homeomorphism, $T\to T$, or $\hat T\to\hat T$, we associate the induced automorphism in $H_1(T)=H_1(\hat T)=\Z^2$.
 
Consider a double covering
$\pi : T \to \DD_3$ branched at the marked points and denote by $\rho:T\to T$ its deck transformation. 
Covering $\pi$ induces an epimorphism
$\Psi_\pi:\Mcg(T)\to\Mcg(\M)$, with the kernel $\{I,[\rho]\}$, where $I$ is the identity and $\rho$ is the deck transformation of the covering, representing $-I\in\SL_2(\Z)$,
which gives an isomorphism $B_3/\DDD^2=\Mcg(\M)=\PSL_2(\Z)$,
 (cf., \cite[Sec. 3.2]{Th}) and identifies $\Psi_\pi$ with the projection $\Psi:\SL_2(\Z)\to\PSL_2(\Z)=\SL_2(\Z)/(-I)$.

The {\it pure mapping class group} $\PMcg(\M)\subset\Mcg(\M)$ is the subgroup preserving marked points.
It is freely generated by Dehn twists $t_c=\s_1^2$ and $t_d=\s_2^2$, where curves $c$, $d$ go, respectively,  around $p_1,p_2$ and $p_2,p_3$, as it is shown in Figure \ref{twocurves}. 
\begin{figure}[hbt]
  \begin{center}
   \includegraphics[width=5cm]{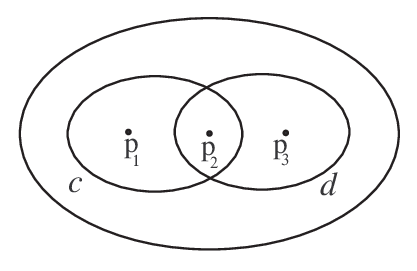}
    \caption{The curves $c$ and $d$ on $\M$}
    \label{twocurves}
  \end{center}
\end{figure}
In addition to an obvious identification $\PMcg(\M)=PB_3/\DDD^2$ with the quotient of the pure braid group, the isomorphism $\Mcg(\M)=\PSL_2(\Z)$ identifies $\PMcg(\M)$
with the congruence subgroup, $\bar\G(2)\subset\G(2)$, defined as follows:
$$
\begin{aligned}
\G(2)=&\{A\in\SL_2(\Z)\,|\,A=\left[\begin{matrix}2k+1&2n  \\2m&2\ell+1 \end{matrix}\right], k,\ell,m,n\in\Z\},\\
\bar\G(2)=&\{\left[\begin{matrix}4k+1&2n  \\2m&4\ell+1 \end{matrix}\right]\in\G(2)\,|\, k,\ell,m,n\in\Z\}.
\end{aligned}
$$
 
As is well-known (e.g., \cite[Prop. 4.4.2]{L}),
$\bar\G(2)$ is freely generated by $U^2$ and $L^2$, where
$U=\left[\begin{matrix}1&1  \\0&1 \end{matrix}\right],\ L=\left[\begin{matrix}1&0  \\1&1 \end{matrix}\right]$.
Moreover: 

\begin{Proposition}\label{congruence-lemma}
\begin{enumerate}\item
$\G(2)=\bar\G(2)\oplus\Z/2$, where $\Z/2$-summand is generated by $-I$. 
\item
$\PMcg(\M)$ is naturally identified with $\bar\G_2$. This identification comes from the restriction of $\Psi$ 
and the identification $\PSL_2(\Z)=\Mcg(\M)$:
$$\bar\G(2)\xrightarrow[\cong]{\Psi} \Psi(\bar\G(2))=\G(2)/(-I)\subset\PSL_2(\Z).\qed$$
\end{enumerate}
\end{Proposition}

 \subsection{Essential curves and multi-curves}\label{multi-curves}
Recall that  an {\it essential curve} in a surface $S$  is a simple closed curve
in the complement of the marked points,
which does not bound a disc without or with only one marked point, and does not bound an annulus together with a boundary component.
A {\it multi-curve} in $S$ 
is a non-empty finite set of disjoint essential curves.
We denote by $\CC(S)$ and by $\CCm(S)$ the sets of isotopy classes of essential curves and, respectively, multi-curves
in $S$.
In the case $S=\M$, an essential curve, $\g$, bounds a disc containing precisely two marked points, $p_i$, $p_j$, $1\le i,j\le3$.
Such a disc can be viewed as a neighbourhood of some {\it core arc}
$r$ connecting $p_i$ and $p_j$.
By an {\it arc} in $\M$, we mean an embedded path connecting distinct marked points. We
denote by $\A(\M)$ the set of isotopy classes of arcs in $\M$.

Note that an arc $r$ can be lifted to the torus $T$ as
a simple closed curve $\til \g=\pi^{-1}(r)$.
The following fact is well-known (cf., \cite[Expos\'e 2,III]{FLP}).

\begin{Lemma}\label{core-arcs}
$(1)$ Assignment of the core arc to an essential curve gives
 a 1--1 correspondence between $\CC(\M)$ and $\A(\M)$.
 
$(2)$ The components of a multi-curve in $\M$ are isotopic to each other.

$(3)$ The pull-back $\pi^{-1}(\g)$ of any essential curve $\g\subset \M$ is a multi-curve with 2 components
isotopic to $\til\g=\pi^{-1}(r)$, for the core arc $r$ of $\g$.

$(4)$ The Dehn twist $t_{\g}\in\PMcg(\M)$ lifts to $t_{\til \g}^2\in \Mcg(T)$.
\qed\end{Lemma}

An essential curve $\g\subset T$ can be taken with any orientation, which gives a pair of classes $\pm[\g]\in H_1(T)$.
If $\g\subset T$ is a multi-curve, its connected components are isotopic and we obtain two classes $\pm[\g]\in H_1(T)$,
as we require that all components of $\g$ are oriented coherently. 

\begin{Lemma}\label{2-1}$($\cite[Sec.2.2.4.]{FM}$)$
Assignment of $\pm[\g]\in H_1(T)$ to $\g\subset T$ yields a bijection between
$\CCm(T)$ and $H_1(T)\sm\{0\}/\{\pm1\}$. It restricts to a bijection between $\CC(T)$ and the set of pairs of opposite
primitive classes in $H_1(T)$.
\qed\end{Lemma}

\begin{Lemma}\!\cite[Prop. 2.6]{FM}\label{torus-disc-corresp}
The map $\g\mapsto\til\g$ is a bijection $\CCm(\M)\to\CCm(T)$.
It restricts to a bijection $\CC(\M)\to\CC(T)$.
\qed\end{Lemma}

\subsection{The Dynnikov coordinate system for $\M$.}
Throughout the paper, we assume fixing an identification of $\M$ with a disc with 3 marked points, shown in Figure \ref{twocurves}, where
we consider arcs $\a_i, \b_i$, $i=1,2$ as shown in Figure  \ref{Dyncoord}.
 \begin{figure}[h!]
  \begin{center}
   \includegraphics[width=6cm]{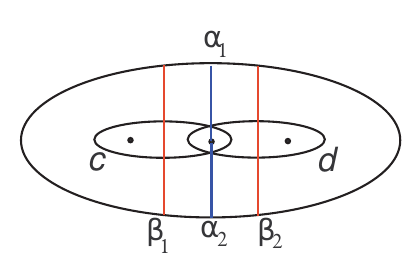}
    \caption{The curves $c$ and $d$ on $\M$}
    \label{Dyncoord}
  \end{center}
\end{figure}

For a multi-curve $\g\subset \M$ we let
$$a(\g)=\frac{\alpha_{2}(\g) - \alpha_{1}(\g)}{2} \quad \text{and} \quad  b(\g)=\frac{\beta_1(\g) - \beta_{2}(\g)}{2},$$
where $\alpha_i(\g), \beta_i(\g)\ge0$ are the {\it geometric intersection numbers} of $\g$ with the corresponding arcs.
By the  geometric intersection number, we mean
the minimum of intersection points for multi-curves representing the isotopy class of $\g$ 
with the corresponding arc, so that we obtain integer functions 
$$\a_i,\b_i:\CCm(\M)\to\Z_{\ge0}\quad\text{and}\quad a,b:\CCm(\M)\to\Z.$$

The integers $a(\g)$, $b(\g)$ will be called
{\it Dynnikov coordinates} of  $\g$ and of its class $[\g]\in\CCm(\M)$ (cf., \cite{D}). 

\begin{Example}\label{Ex1} 
The Dynnikov coordinates of curves $c$, $d$, and $e$ in  Figure\,\ref{threecurves},
\begin{figure}[h!]
  \begin{center}
   \includegraphics[width=5cm]{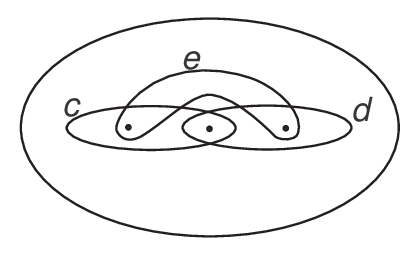}
    \caption{}
    \label{threecurves}
  \end{center}
\end{figure}
are $(0,1)$, $(0,-1)$, and $(-1,0)$, respectively.
The set of classes of these curves (denoted by the same letters as curves themselves) will be denoted
 $\D = \{c, d, e\}\subset\CC(\M)$.
\end{Example}

The  importance of Dynnikov coordinates is explained, in particular, by the following correspondence (known for $n\ge3$ marked points in a disc too).

\begin{Theorem}\label{Dyn-corresp} $($\cite{D}, \cite{HY}$)$
The mapping $\g\mapsto (a(\g),b(\g))$ defines a bijective correspondence 
$\CCm(\M) \to \mathbb{Z}^{2}\sm \{0\}.$
\end{Theorem}

Using the correspondences in Lemmas \ref{core-arcs}, \ref{torus-disc-corresp} and Theorem \ref{Dyn-corresp},
we can characterize
any multi-curve $\g_T\subset T$ with its {\it Dynnikov coordinates} $a=a(\g_T)$ and $b=b(\g_T)$, which are by definition
the Dynnikov coordinates of the multi-curve $\g_{\M}\subset \M$ such that $\g_T=\til\g_{\M}$.

\subsection{Geometric intersection numbers after lifting of curves} 

\begin{Lemma}\label{number-coincidence}
Suppose that $s\subset \DD_3$ is one of the arcs $\a_i$, $\b_i$, $i=1,2$ in Figure \ref{Dyncoord},
 $\g\subset \M$ is an essential curve and $\til\g\subset T$ is its lifting as above. Then the following geometric intersection numbers in $\M$ and in $T$ coincide:
$$\g\cdot s=\til\g\cdot \til s,\quad\text{ where } \til s=\pi^{-1}(s).$$
\end{Lemma}

\begin{proof} 
$\g\cdot s=\frac12\pi^{-1}(\g)\cdot \til s=\til\g\cdot \til s$, since 
$\pi^{-1}(\g)$ is formed by 2 components isotopic to $\til\g$.
\end{proof}

\subsection{The update rules.}\label{update}
The natural $B_3$-action on $\CCm(\M)$ via projection $B_3\to\Mcg(\M)$ defines a $B_3$-action on the {\it Dynnikov plane} $\Z^2$, 
known as the {\it update rules}
(see, e.g., \cite{D,HY,Th}).

\begin{Proposition}$($\cite[Lemmas 4-5]{HY}$)$
A multi-curve $\g\subset \M$ with Dynnikov coordinates $(a,b)\in\Z^2\sm\{0\}$ is sent by
the actions of $\s_1^{\pm1}$ and $\s_2^{\pm1}$ to a multi-curve $\g'$ with coordinates $(a',b')$ as is described in the Table \ref{update-tab}.
\begin{table}[h!]
\caption{The update rules}\label{update-tab}
\begin{tabular}{cccc}
\hline
$\s_1$-action&$\s_2$-action\\
\hline
\vbox{
\hbox{$a' = a + b - max\{0,a,b\}$}
\hbox{$b' = max \{b, 0 \} - a. $}
}
&
\vbox{
\hbox{$a' = max\{a + max \{0,b\}, b \}$}
\hbox{$b' = b - (a + max \{0, b \}$)}
}\\
\hline
\hline
$\s_1^{-1}$-action&$\s_2^{-1}$-action\\
\hline
\vbox{
\hbox{$a' =  max\{0, a + max \{0, b\}\} - b$}
\hbox{$b' = a + max \{0, b \}$}
}
&
\vbox{
\hbox{$a' = a - max\{a + b, 0, b\} \}$}
\hbox{$b' = a + b - max\{0, b\}$}
}\\
\hline
\end{tabular}
\end{table}
\qed \end{Proposition}
 
\section{The Dynnikov Dynamics of the basic Dehn-twists}\label{dynamics}

\subsection{Geometry of the action of $t_c$ and $t_d$ on the Dynnikov plane}
\begin{Theorem}\label{tc-td-action}
$(1)$ The action of Dehn twists $t_c$ and $t_d$ on the Dynnikov coordinate plane is defined by the following piecewise linear functions:
\begin{equation}\label{Eq1}
t_c(a,b)=\left\{
\begin{array}{llll} 
(b-a, -b)& {\rm~ if }\; a \geq 0, b \leq a \qquad \ \ \text{Region A}\\ 
(b-a,b-2a) & {\rm~ if }\; 0 \leq a \leq b \leq 2a \quad \ \text{Region B}\\ 
(a, b-2a)& {\rm~ if }\; 2a \leq b, b \geq 0 \qquad \ \text{Region C}\\
(a+b,-2a-b)& {\rm~ if }\; a \leq 0, b \leq 0 \qquad\ \ \text{Region D}\\  
\end{array} 
\right.
\end{equation}
\begin{equation}\label{Eq2}
t_d(a,b)
= \left\{
\begin{array}{ll}  
(b-a,-b) & {\rm~ if }\; a \leq 0, a \leq b \qquad \ \ \text{Region -A}\\ 
(b-a,b-2a)& {\rm~ if }\; 2a \leq b \leq a \leq 0 \quad \ \text{Region -B}\\ 
(a,b-2a)& {\rm~ if }\;  b \leq 2a, b \leq 0 \qquad \ \text{Region -C}\\ 
(a+b, -2a-b)& {\rm~ if }\; a \geq 0, b \geq 0 \qquad\ \ \text{Region -D}\\
\end{array} 
\right.
\end{equation}
\begin{figure}[h!]
\begin{center}
\hskip-4mm     \includegraphics[width=6.2cm]{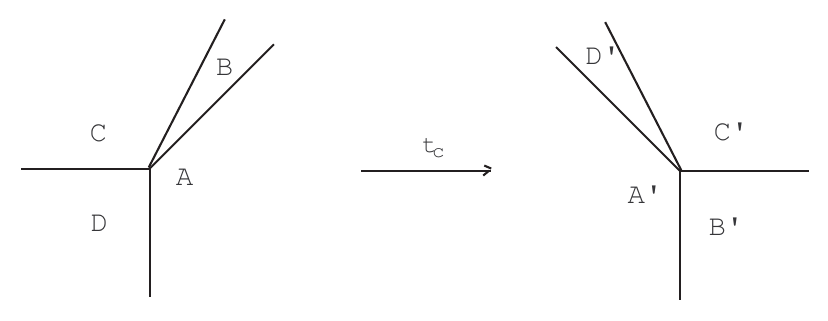}\hskip2mm
   \includegraphics[width=6.2cm]{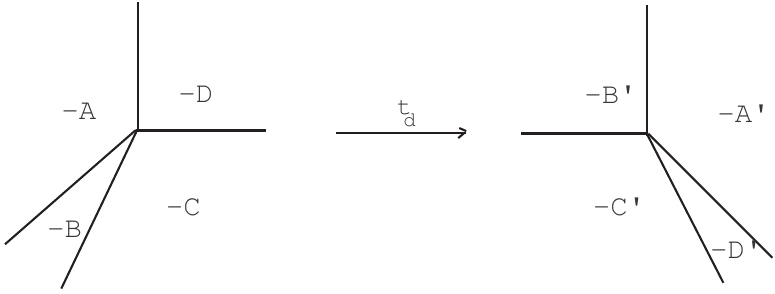}
    \caption{The linearity regions for the $t_c$ and $t_d$-action}
    \label{Fan1-2}
  \end{center}
\end{figure}

$(2)$ Figure \ref{Fan1-2} shows the four linearity regions for each transformation, $t_c$ and $t_d$,
and the images of these regions.
Namely, for $t_c$, the image of region $A$ is $A'$, etc. For $t_d$ the linearity regions and their images turn out to be centrally symmetric to those of $t_c$, so that
$-A$ is mapped to $-A'$, etc.
\end{Theorem}

\begin{proof}
(1) Formulas (1)--(2) with some non-significant modifications were earlier obtained in \cite{ADMY}.
To obtain them, it is sufficient to apply twice the update rules from Section\,\ref{update} and
express the action of $\s_1$ and $\s_2$ in the corresponding linearity regions.
\newline
(2) These regions  are trivially determined by the inequalities in (1)--(2).
\end{proof}

\begin{Remark} The vertical positive ray in Figure \ref{orbits}(a)
(negative ray on Figure \ref{orbits}(b))
contains fixed points of the $t_c$-action (respectively, of $t_d$-action).
\end{Remark}

\begin{Remark}
Bijectivity of $t_c$ and $t_d$ on  $\Z^2$ 
implies that they are area-preserving on $\R^2$.
Together with the correspondence between the regions indicated in Figure \ref{Fan1-2}
this determines the formulas (1)--(2) of Theorem \ref{tc-td-action}.
\end{Remark}

\subsection{Tracks of orbits of $t_c$ and $t_d$ on the Dynnikov plane}\label{tracks}
The formulas (1)--(2) of Theorem \ref{tc-td-action} imply straightforwardly the following.

\begin{Corollary}
The orbits of the action by $t_c$ and $t_d$ are contained in the
broken-lines shown in Figure\,\ref{orbits} (with the clockwise direction for both actions). 
\begin{figure}[h!]
  \begin{center}
   \includegraphics[width=12cm]{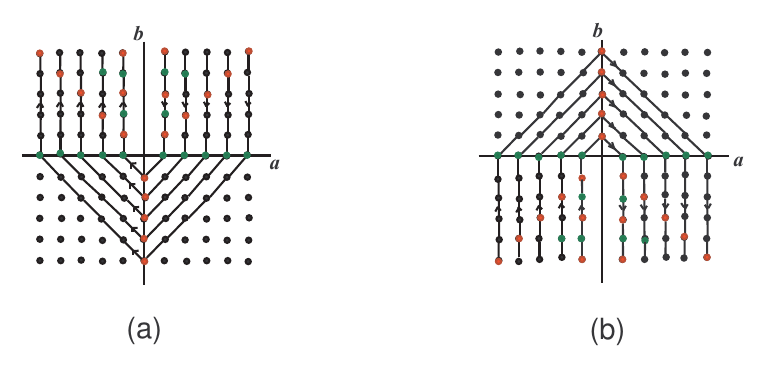}
    \caption{Tracks (a) $O^c_n$ of $t_c$-action and (b) $O^d_n$ of $t_d$-action}
    \label{orbits}
  \end{center}
\end{figure}
\qed\end{Corollary}

We denote by $O^c_n$ and $O^d_n$, $n>0$, the tracks containing orbits of $t_c$ and $t_d$-action respectively (see Figure \ref{orbits}), and 
passing through the point $(n,0)\in\Z^2$.

\begin{Proposition}\label{3-orbits}
The natural action of $\PMcg(\M)$ on $\CC(\M)$ has precisely 3 orbits represented by the curves $c,d$, and $e$ in Figure \ref{threecurves}.
\end{Proposition}
\begin{proof} 
If curves $\g_1$ and $\g_2$ enclose the same pair of marked points, then there exists a homeomorphism $f:\M\to\M$ sending $\g_1$ to $\g_2$
and preserving marked points. Thus, an orbit of $[\g]\in\CC(\M)$ is determined by the two marked points located inside the curve $\g$.
\end{proof}

\subsection{The action of $t_e$ and $t_{e'}$ on the Dynnikov plane}
By analogy with Theorem \ref{tc-td-action}, we can give
a description of the action of Dehn twists
around the curves $e,e'$ shown in Figure \ref{2curves} (the proof is similar, via
update rules).

\begin{figure}[h!]
  \begin{center}
   \includegraphics[width=4cm]{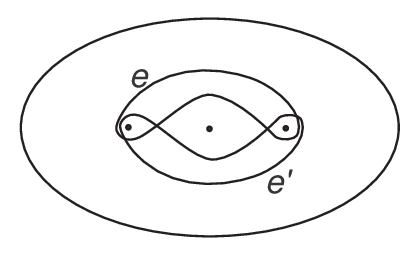}
    \caption{Curves $e$ and $e'$}
    \label{2curves}
  \end{center}
\end{figure}

\begin{Theorem}\label{e-action}
$(1)$ The action of Dehn twists $t_e$ and $t_{e'}$ on the Dynnikov coordinate plane is defined by the following piecewise linear functions:
\[
t_e(a,b)
= \left\{
\begin{array}{llllllll} (-3a-b, -2a-b)& {\rm~ if }\; a \geq 0, b \geq 0&A\\ 
(-3a+2b,-2a+b) & {\rm~ if }\; a \geq 0, b \leq 0&B\\ 
(a+2b,b)& {\rm~ if }\; a \leq 0, b \leq \frac{-a}{2}&C\\ 
(a+2b, -2a-3b)& {\rm~ if }\; 0 \leq \frac{-a}{2}  \leq  b \leq \frac{-2a}{3}&D\\ 
(-a-b,-2a-3b)& {\rm~ if }\; 0 \leq \frac{-2a}{3} \leq b \leq -a&E\\ 
(-a-b,-b)& {\rm~ if }\; 0 \leq -a \leq b &F\\ 
\end{array} \right.
\]

\[
t_{e'}(a,b)= \left\{
\begin{array}{llllllll} (-3a-b, -2a-b)& {\rm~ if }\; a \leq 0, b \leq 0&-A\\ 
(-3a+2b,-2a+b) & {\rm~ if }\; a \leq 0, b \geq 0&-B\\ 
(a+2b,b)& {\rm~ if }\; a \geq 0, b \geq \frac{-a}{2}&-C\\ 
(a+2b, -2a-3b)& {\rm~ if }\; \frac{-2a}{3} \leq b \leq \frac{-a}{2} \leq 0&-D\\ 
(-a-b,-2a-3b)& {\rm~ if }\; -a \leq b \leq \frac{-2a}{3} \leq 0&-E\\ 
(-a-b,-b)& {\rm~ if }\; b \leq -a \leq 0 &-F\\ 
\end{array} \right.
\]

$(2)$ Figure \ref{fan3-4} shows the six linearity regions for each transformation, $t_e$ and $t_{e'}$,
and the images of these regions.
%
\begin{figure}[h!]
\hskip-5mm \includegraphics[width=6.4cm]{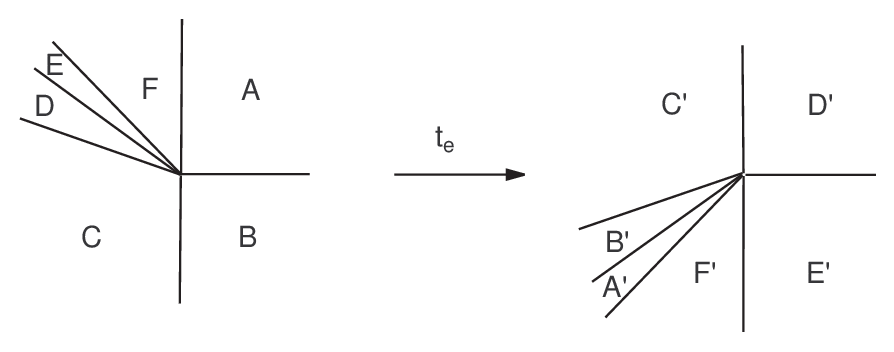}\hskip3mm    \includegraphics[width=6.4cm]{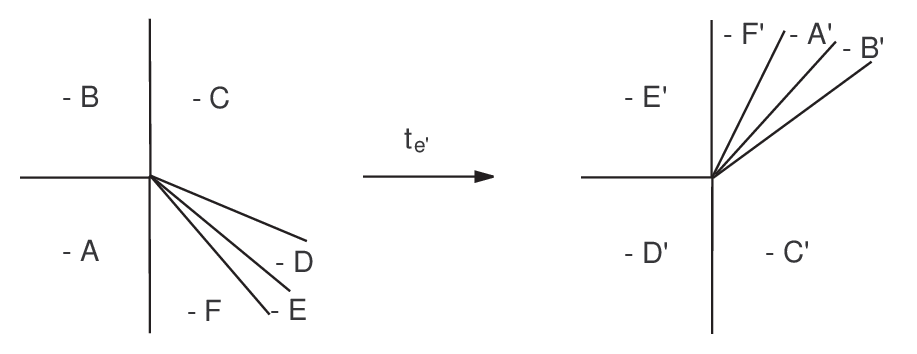}
    \caption{Linearity regions for the action of $t_e$ and $t_{e'}$. Labeling of the regions follow the same conventions as in 
    Theorem \ref{tc-td-action}(2) and Figure \ref{Fan1-2}        } \label{fan3-4}
\end{figure}
\end{Theorem}


\subsection{The orbits of $t_e$ and $t_{e'}$}
Like in Section \ref{tracks} we show in Figure \ref{orbits2}
the tracks $O_n^{e}$ and $O_n^{e'}$
containing $t_e$ and $t_{e'}$-orbits respectively. The indices $n$ refer to the intersection point of the track with the positive ray of $b$-axis.

\begin{figure}[h!]
  \begin{center}
   \includegraphics[width=10cm]{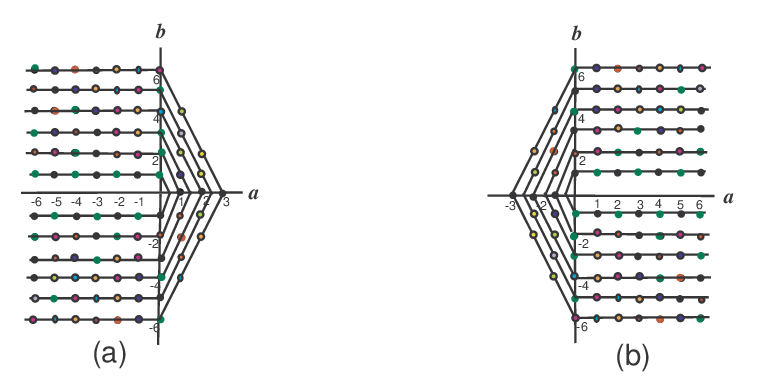}
    \caption{Tracks (a) $O_n^e$ of $t_e$-action (b)  $O_n^{e'}$of $t_{e'}$-action}
    \label{orbits2}
  \end{center}
\end{figure}

These tracks are broken lines which are 
the images of the lines $p-q=n$ and $p+q=n$, respectively, for $n\in\Z$, on the torus $(p,q)$-coordinate plane under the projection to the 
Dynnikov $(a,b)$-coordinate plane by the map $(p,q)\mapsto(a,b)$ in Theorem \ref{transform-thm}.
The action of $t_e$ on $(a,b)\in O_n^e$
consists in $2n$ consecutive integer jumps in the clockwise direction along the track $O^e_n$.
The action of $t_{e'}$ on $(a,b)\in O_n^{e'}$ is a similar 2n-step jump along $O_n^{e'}$ in the clockwise direction.

\section{The Dynnikov coordinates on a torus}\label{torus-disk-sec}

\subsection{From torus coordinates to Dynnikov coordinates}\label{torus-Dyn}
\begin{Theorem}\label{transform-thm}
The Dynnikov coordinates $(a,b)$
of a multi-curve $\g_T\subset T$ with torus coordinates $(p,q)$ are
$$(a,b)= (\frac{|p-q| - |p+q|}{2}, |p|-|q|).$$
\end{Theorem}

\begin{Remark}
Note that $a=\frac{|p-q| - |p+q|}{2}$ implies that $|a|=\min(|p|,|q|)$ and the sign of $a$ is opposite to that of the product $pq$.
\end{Remark}

\begin{proof}
By Lemma \ref{torus-disc-corresp}, we can pick $\g_{\M}\subset \M$, such that a covering curve $\g_T$
(as defined in Sect. \ref{multi-curves}) is isotopic to $\til\g_{\M}$.
Let us consider a {\it flat torus} model of $T$ shown on the leftmost sketch of Figure \ref{torusone holed},
with a hole around the corner-vertex.
\begin{figure}[hbt]
   \includegraphics[width=8.5cm]{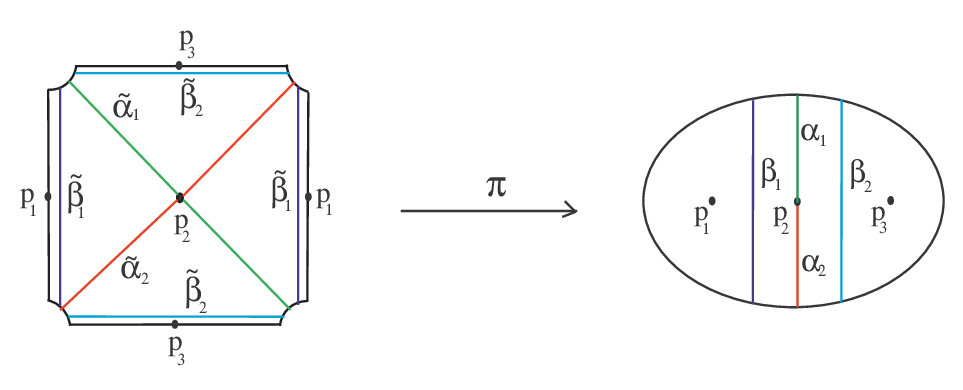}\hskip10mm
   \includegraphics[width=3cm]{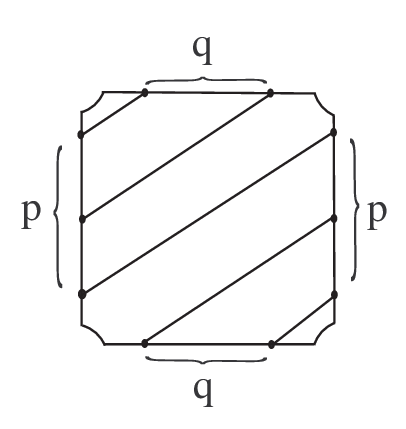}
    \caption{}
    \label{torusone holed}
\end{figure}
The branched double covering $\pi$ can be represented so that
the ramification points $p_1$ and $p_3$ are  the midpoints of the vertical and horizontal sides respectively, while $p_2$ is at the center.
The covering curves $\til c$, $\til d$
are represented then by the middle horizontal and vertical lines (axes of symmetry) of the flat torus.
Curves $\til\g=p\til c+q\til d$, which, by definition, are of homology class $(p,q)\in\Z^2=H_1(T)$, after we straighten them by an isotopy, will have constant slope $\frac{q}p$
(see the rightmost sketch in Figure \ref{torusone holed}) and will intersect both $\til c$ and the horizontal sides of the flat torus at the same minimal possible number of points, $q$,
while both $\til d$ and the vertical sides will be intersected at $p$ points.

The diagonal axes of symmetry on the leftmost diagram are $\til\a_1=\pi^{-1}(\a_1)$ (the main diagonal), and  $\til\a_2=\pi^{-1}(\a_2)$.
The sides, vertical and horizontal,
slightly shifted inside the flat torus, as is shown in Figure \ref{torusone holed}, are 2-component multi-curves $\pi^{-1}(\b_i)$, for $i=1$ and $i=2$ respectively.

The geometric intersections of $\til\g$ with $\til \a_1$ and $\til\a_2$
are clearly $|p+q|$ and $|p-q|$, while with component $\til \b_i$, $i=1,2$, the geometric intersection numbers are $|p|$ and $|q|$ respectively.
Now it is left to use Lemma \ref{number-coincidence} giving 
\[
\pushQED{\qed} \a_1(\g_{\M})=|p+q|,\quad \a_2(\g_{\M})=|p-q|,\quad \b_1(\g_{\M})=2|p|,\quad \b_2(\g_{\M})=2|q|.
\qedhere
\]
\end{proof}

Let us denote by  $\Phi:\Z^2\to\Z^2$ the coordinate transformation map of Theorem \ref{transform-thm},
sending $(p,q)$ to $(a,b)=(\frac{|p-q|-|p+q|}2,|p|-|q|)$, and consider its extension $\Phi_\R:\R^2\to\R^2$ defined by the same formulas.

\begin{figure}[hbt]
  \begin{center}
   \includegraphics[width=10cm]{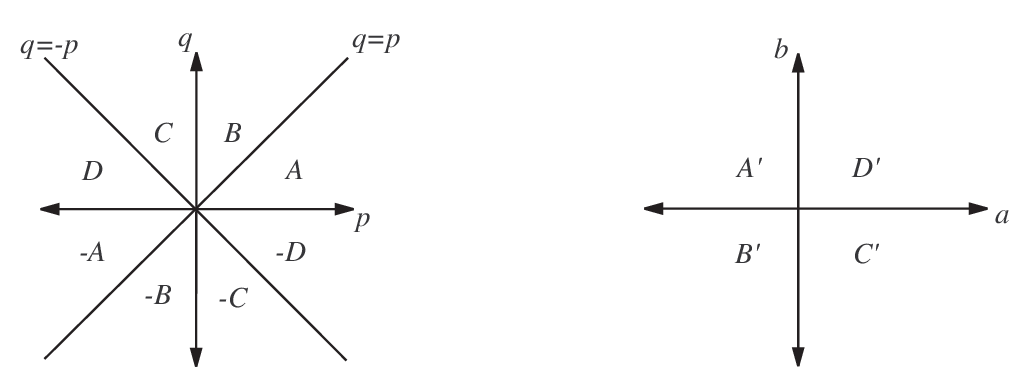}
    \caption{Eight sectors of the leftmost figure are linearly transformed into quadrants of the rightmost figure, so that each sector, $X$ and $-X$, is sent to $X'$, for $X=A,B,C,D$}
    \label{sectors}
  \end{center}
\end{figure}
\begin{Theorem}\label{Dynnikov-properties}
The maps $\Phi$ and $\Phi_\R$ have the following properties:
\begin{enumerate}\item
$\Phi_\R$ is piecewise-linear and area preserving.
\item
The map $\Phi_\R$ is a double covering branched at the origin, and is {\it even} in the sense that $\Phi_\R(-(p,q))=\Phi_\R((p,q))$.
\item
$\Phi_\R$ sends each of the 8 regions of $\R^2$ bounded by the axes $p$, $q$ and angle bisectors $p=\pm q$ onto a quadrant of $\R^2$ linearly, as is indicated in
 Figure \ref{sectors}.
\item
$\Phi(\Z^2)=\Z^2$, and so, $\Phi$ defines a 1--1 correspondence between $\Z^2/(-1)$ and $\Z^2$, where $\Z^2/(-1)$ stands for the quotient
by the involution $(p,q)\mapsto-(p,q)$.
\end{enumerate}
\end{Theorem}

\begin{proof}
All follow straightforwardly from the formulas for $\Phi$.
\end{proof}

\begin{Corollary}
For any $n\in\Z\sm\{0\}$,
the map $\Phi$ sends horizontal line $q=n$ to the orbit $O^c_{|n|}$
and vertical line $p=n$ to the orbit $O^d_{|n|}$ (see
Figure \ref{orbits}).
\qed\end{Corollary}

\subsection{From Dynnikov to torus coordinates}
The following Proposition gives
inversion formulas for
 the map 
 in Theorem \ref{transform-thm}.

\begin{Proposition}\label{reversion}
The inverse transformation $(a,b)\mapsto\pm(p,q)$ to the transformation in Theorem \ref{transform-thm} is as follows.
$$\pm(p,q)=\begin{cases} (|a|+b,-a), \text{ if } b\ge0\\
(a,b-|a|), \text{ if } b\le0
\end{cases}.$$
\end{Proposition}

\begin{proof}
It is straightforward, for instance, from the description of $\Phi$ in Theorem \ref{Dynnikov-properties}, as we take the inverse linear maps in each linearity region
depicted on Figure \ref{sectors}.
\end{proof}

\section{A minimal path algorithm}\label{algorithm-sec}

\subsection{Inductive step for untwisting}\label{algorithm}
For any curve $\g\in\CC(\M)$ we describe a minimal sequence $\g_0,\dots,\g_r$, which starts at
$\g=\g_0$ and ends up at
 $\g_r\in \D$.
Namely, we act on $\g_k$, $k\ge0$,
by some generator $t\in\{t_c^{\pm1},t_d^{\pm1}\}$ to obtain $\g_{k+1}$.
 Such $t$ is determined by the Dynnikov coordinates $(a_k,b_k)\in\Z^2$ of $\g_k$, for non-vanishing $a_k,b_k$, as follows:

\begin{enumerate}
\item
$t=t_c\ \text{ if } a_k,b_k>0$,
\item
$t=t_c^{-1}\ \text{ if } a_k<0,b_k>0$,
\item
$t=t_d\ \text{ if } a_k,b_k<0$,
\item
$t=t_d^{-1}\ \text{ if } a_k>0,b_k<0$.
\end{enumerate}

Since vectors $(a_k,b_k)$ representing curves in $\CC(\M)$ are primitive, vanishing of $a_k$ or $b_k$
implies that $(a_k,b_k)\in\{(0,\pm1),(\pm1,0)\}$. The three vectors different from $(1,0)$ represent elements of
$\D=\{c,d,e\}$, which means termination of the process.
In the case of vector $(1,0)$, we need one extra twist, either $t=t_c$ or $t=t_d^{-1}$, which gives $(-1,0)$ representing $e$.
Among the two possible choices we make preference to the same twist as was used on the previous step
(note that $t_c^{-1}$ or $t_d$ could not appear on the previous step, since they lead to $(1,0)$ from $(-1,0)$, one of our terminal points).
Thus, uncertainty is resolved, except for the case $(1,0)$ being the initial vector.

\begin{Remark}
To find the coordinates $(a_{k+1},b_{k+1})$ after the action of $t$ specified in the above rule, we can use formulas of Theorem \ref{tc-td-action}.
But it is more practical to jump along the orbits of $t_c$ and $t_d$-action shown in Figure \ref{orbits}.
Namely, the action of $t_c$ (respectively, $t_d$) on vector $(a,b)$ from the track $O_n^c$ (respectively, from $O_n^d$)
consists in $2n$ jumps along the corresponding orbit in the clockwise direction, where a {\it jump} means passing to the next integral point of the orbit.
The action of $t_c^{-1}$ and $t_d^{-1}$ is described by similar jumps along the same orbits, but in the opposite direction.
\end{Remark}

\subsection{The Cayley graph of $\PMcg(\M)$-action and the distance in it}
As we define the {\it Cayley graph of a group action}, $G$ on a set $V$ associated with
a set of generating elements $g_1,\dots,g_n\subset G$, we take $V$ as the set of vertices and connect a pair of {\it distinct}
vertices $v_1,v_2\in V$ by edge if $v_2=g_i(v_1)$ or $v_1=g_i(v_2)$ for some $i$.
Note that, in our definition, such a graph has no loop-edges. If two vertices belong to the same connected component (one orbit of $G$-action),
we consider the {\it graph-distance}, that is, the length
of a minimal path between the vertices.

Let us consider the Cayley graph $\G_\CC$ of $\PMcg(\M)$-action on $\CC(\M)$, associated with the generators $t_c,t_d$.
By Proposition  \ref{3-orbits}, it has 3 connected components and, thus,
the function $\CC(\M)\to\Z_{\ge0}$, 
measuring the graph-distance from a vertex $\g\in\CC(\M)$ to
the set $\D=\{c,d,e\}$ is well-defined.
By the definition of Cayley graph, it is equal to
the minimal number of basic Dehn twists, $t_c$, $t_d$, or their inverse, required to
transform $\g$ to one of the pattern curves, $c$, $d$, or $e$.

\begin{Theorem}\label{minimal-Dynnikov}
The number of steps in the algorithm in Section \ref{algorithm} is equal to the minimal distance in $\G_\CC$
from $\g\in\CC(\M)$ to the set $\D=\{c,d,e\}$.
\end{Theorem}

A proof is postponed untill Section \ref{comparison-sec}.

\subsection{Examples}

\begin{Example} Applying our algorithm to the vertex $\g\in\CC(\M)$ with coordinates $(a,b)=(10,3)$, we obtain a path of length 5 in $\G_\CC$:

\begin{align*}
(10,3) \stackrel{t_c}{\longrightarrow} (-7,-3)  \stackrel{t_d}{\longrightarrow} (4,3)
\stackrel{t_c}{\longrightarrow} (-1,-3) \stackrel{t_d}{\longrightarrow} (-1,-1)
\stackrel{t_d}{\longrightarrow} (0,1).
\end{align*}

For instance, at the first step, since $(10,3)$ belongs to the quadrant $a,b>0$, we apply $t_c$, which makes
$20$ lattice jumps along the tracks $O^c_{10}$ containing $(10,3)$: the first 3 jumps lead to $(10,0)$, the next 10 jumps to $(0,-10)$
and the last 7 jumps lead to $(-7,-3)$, the next vertex of our path. Following further our algorithm, we end up at point $(0,1)$ representing $c\in\D$.

\end{Example} 

\begin{Example}
Starting with $(3,10)$, we obtain by our algorithm a 4-path:
\begin{align*}
(3,10) \stackrel{t_c}{\longrightarrow} (3,4)  \stackrel{t_c}{\longrightarrow} (1,-2)
\stackrel{t_d^{-1}}{\longrightarrow} (1,0) \stackrel{t_d^{-1}}{\longrightarrow} (-1,0).
\end{align*}
\end{Example}

\subsection{Termination of the algorithm}

\begin{Lemma}\label{termination}
For any initial  essential curve $\g=\g_0$,
the sequence $\g_k$ described by the above algorithm terminates at some element of $\D=\{c,d,e\}$.
\end{Lemma}

\begin{proof}
It is enough to observe that the index $n$ of the track $O_n^c$ or $O_n^d$ along which we perform $2n$ jumps in the algorithm,
for a decreasing sequence,
and the index $n$ decreases strictly if the type of step (1)--(4) differs from the type of the previous step.
Since we can perform each step only finitely many times, we will end up at some curve $\g_r$ from the track $O_1^c$ or $O_1^d$.
\end{proof}

\section{A digression on the even integer continued fractions}\label{digression}
The algorithm from Section \ref{algorithm} can be viewed as a transformation to  the Dynnikov coordinates
of Euclid's algorithm for finding even continued fractions, performed in the torus coordinates.
We review below some basic (cf. \cite{Sch}) facts about such continued fractions.

\subsection{Generalities}\label{generalities-Euclid}
As is known, Euclid's algorithm for finding the greatest common divisor $(m,n)=1$ 
of coprime integers $m,n>0$ gives consecutive fractions $q_i$ that form
a {\it positive continued fraction presentation}
$$\frac{m}n=
q_0+\cfrac{1}{q_1+\cfrac{1}{q_2+
\cfrac{1}{
  \ddots
  \vphantom{\cfrac{1}{1}} 
  \,
  +\cfrac{1}{q_r}
  }}}
  $$
 where $q_0\ge0$,  $q_1,\dots,q_r> 0$. It is denoted shortly $\frac{m}n=[q_0,\dots,q_r]$.
Such presentation is unique if 
we do not allow $q_r=1$ for $\frac{m}n\ne1$.

Recall also that a positive continued fraction presentation $\frac{m}n=[q_0,\dots,q_r]$ for $m,n>0$ is equivalent to a factorization
$\left[\begin{matrix} m \\ n\end{matrix}\right]=U^{-q_0}L^{-q_1}U^{-q_2}\dots v$, into transvections $U,L$ introduced in Sec. \ref{modular-subsect}.
Here $v$ is $\left[\begin{matrix} 1 \\ 0\end{matrix}\right]$ if $r$ is odd and $\left[\begin{matrix} 0 \\ 1\end{matrix}\right]$ if even.
Such a parity rule for $v$ is due to the fact that odd (respectively, even) steps of Euclid's algorithm are interpreted as multiplication by the corresponding power of $U$ (respectively, $L$).

\subsection{A modified Euclid's algorithm}

\begin{Proposition}\label{spin-frac}
Consider $\frac{m}n\in\Q$, with $(m,n)=1$.
\begin{enumerate}\item
If $mn$ is even, then there exists
a continued fraction presentation $\frac{m}n=[q_0,\dots,q_r]$ with
$q_i\in2\Z$, for $i\ge0$, with $q_i\ne0$ for $i>0$, and
$r=m\,\rm{mod}\,2$.
\item
If $mn$ is odd,
then  there 
exists a presentation $\frac{m}n=[q_0,\dots,q_r,1]$ with 
$q_i\in2\Z$, for $i\ge0$, with
$q_i\ne0$ for $i>0$.
Besides, in such presentation $q_r\ne-2$ unless $mn=-1$. 
\end{enumerate}
\end{Proposition}

\begin{proof}
On the first step, if $mn\ne\pm1$,
we perform a division giving $m=q_0n+m_1$, with even $q_0$ and $|m_1|<|n|$.
Next division gives even $q_1\ne0$, so that $n=q_1m_1+n_1$, with $|n_1|<|m_1|$. 
Then we pass to the pair $(m_1,n_1)$, and alternate the order of division
like in the usual Euclid's algorithm.
The pairs of numbers at each step remain coprime,
and in addition,
the parities of $m,m_1,\dots$ remain the same,
as well as the parities of $n,n_1,\dots$,
 since all $q_i$'s are even.

The inductive process stops at one of the following terminal pairs.

$\bullet$ {\it If  $m$ is even and $n$ is odd}: then we reach pair $(0,\e)$, where $\e=\pm1$ is such that  $n=\e\,{\rm mod}\,4$,
after an odd number of steps (that is, obtain a sequence $q_0,\dots,q_r$ with even $r$).
The congruence $n=\e\,{\rm mod}\,4$ follows from $n=n_1=\dots\,{\rm (mod}\,4)$, due to $4|q_im_i$ for all $i$.

$\bullet$ {\it If $m$ is odd and $n$ is even}: then we reach pair $(\e,0)$, where $\e=\pm1$ is such that $m=\e\,{\rm (mod}\,4)$,
after an even number of steps (that is, $r$ will be odd), for similar reasons.

$\bullet$ {\it If $mn=1$}, then the process does not start and $\frac{m}n=1=[1]=[0,1]$.

$\bullet$ {\it If $mn=-1$}, then $\frac{m}n=-1=[-2,1]$.

$\bullet$ {\it If $mn$ is odd, but not $\pm1$}, then at some step we obtain a pair $(k,\e)$ or $(\e,k)$ with $\e=\pm1$ and some odd $k\ne\pm1$.
Then, the next step gives a terminal pair $(\e,\e)$ with $q_r=\frac{k-\e}{\e}$.
Note that $q_r\ne-2$, since otherwise $k=q_r\e+\e=-\e$.
\end{proof}

\subsection{The associated matrix decompositions}
The relation between the continued fractions and matrix factorization 
implies the following.
 
 \begin{Corollary}\label{cor-matrices}
 Consider $\frac{m}n\in\Q$, $(m,n)=1$. Then by Proposition \ref{spin-frac}
 \newline
$(1)$
If $mn$ is even, $\frac{m}n=[q_0,\dots,q_r]$ with all $q_i\in2\Z$, then:
\begin{itemize}\item
 for $m$  even (thus, $n$ odd and $r$ even)
$\left[\begin{matrix} m \\ n\end{matrix}\right]=U^{-q_0}L^{-q_1}\dots U^{-q_r}\left[\begin{matrix} 0 \\ \e\end{matrix}\right],$
where $\e\in\{1,-1\}$ is chosen to be congruent to $n$ modulo $4$.
\item
  for $m$  odd (thus, $n$ even and $r$ odd)
$\left[\begin{matrix} m \\ n\end{matrix}\right]=U^{-q_0}L^{-q_1}\dots L^{-q_r}\left[\begin{matrix} \e \\ 0\end{matrix}\right],$
where $\e\in\{1,-1\}$ is chosen to be congruent to $m$ modulo $4$.
\end{itemize}
$(2)$
If $mn$ is odd, $\frac{m}n=[q_0,\dots,q_r,1]$ with all $q_i\in2\Z$, then
$$\left[\begin{matrix} m \\ n\end{matrix}\right]=U^{-q_0}L^{-q_1}\dots X^{-q_r}\left[\begin{matrix} \e \\ \e\end{matrix}\right],$$
for some $\e\in\{1,-1\}$. Here, $X=U$ if $r$ is even and $X=L$ if odd.
 \qed\end{Corollary}

\subsection{The orbits of the $\bar\G(2)$-action}
Consider the group action of $\bar\G(2)\subset\SL_2(\Z)$ on the set of primitive vectors $\til\P\subset \Z^2$. It can be identified with the 
$\bar\G(2)$-action on the oriented essential curves in $T$, or with the
$\PMcg(\M)$-action
on the
essential curves in $\M$. Corollary \ref{cor-matrices} implies the following.

\begin{Corollary}\label{5orbits}
Group action of $\bar\G(2)$ on $\til\P\subset\Z^2$ has $5$ orbits. 
\begin{enumerate}\item
$\{(p,q)\in\til\P\,| p\in2\Z, q=1\,{\rm mod}\,4\}$, $\{(p,q)\in\til\P\,| p\in2\Z, q=-1\,{\rm mod}\,4\}$,
\item
$\{(p,q)\in\til\P\,| q\in2\Z, p=1\,{\rm mod}\,4\}$, $\{(p,q)\in\til\P\,| q\in2\Z, p=-1\,{\rm mod}\,4\}$,
\item
$\{(p,q)\in\til\P\,| p,q \text{ are both odd\}}$.
\qed\end{enumerate}
\end{Corollary}

The connected components of the Cayley graph, $\G_{\til\P}$, of the $\bar\G(2)$-action on $\til\P$ are in an obvious correspondence with the above 5 orbits.

\begin{Proposition}\label{trees}
The components of $\G_{\til\P}$ representing the first $4$ orbits in Corollary \ref{5orbits} are trees.
The fifth component has a unique
simple cycle (not containing repeated edges) shown in the leftmost sketch of Figure \ref{cycles-fig}.
\end{Proposition}

\begin{figure}[h!]
  \begin{center}
   \includegraphics[width=10cm]{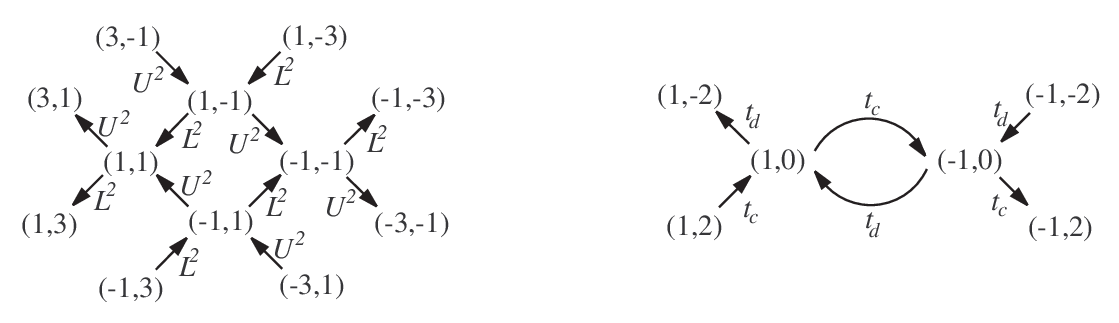}
    \caption{The only simple cycle of the orbit of $v=(1,1)$ in the leftmost sketch (with respect to the torus coordinates) projects by a double covering to the
only simple cycle of the orbit of $e=(-1,0)$ in the Dynnikov plane
}
    \label{cycles-fig}
  \end{center}
\end{figure}

\begin{proof}
We will prove that the component-orbit of $(1,0)$ is a tree and skip the cases of orbits of $(-1,0)$ and $(0,\pm1)$ as analogous ones.
Consider any path-loop, $\lambda$, in $\G_{\til\P}$ based at $v=(1,0)$. It represents a matrix $A\in\bar\G(2)$ obtained by taking product of
$U^{\pm2}$ and $L^{\pm2}$ corresponding to the edges of $\lambda$, in the order of the edges in $\lambda$.
Then $Av=v$, which implies that $A=U^{2m}$, $m\in\Z$. Since $\G_{\til\P}$ is freely generated by $U^2$ and $L^2$ (see Sect. \ref{modular-subsect}),
 the loop $\lambda$ is contractible to the loop represented by $U^{2m}$, which is trivial, since we excluded loop-edges represented by $U^{\pm2}$ in the definition of 
 $\G_{\til\P}$.

In the case of the fifth component, we will similarly look for $A\in\bar\G(2)$ representing a given path-loop $\lambda$ based at $v=(1,1)$.
The condition $Av=v$ trivially implies $A_n=\left[\begin{matrix} 4n+1&-4n \\ 4n&1-4n\end{matrix}\right]$, for some $n\in\Z$.
Observing that $A_n=A_1^n$, where $A_1=\left[\begin{matrix} 5&-4 \\ 4&-3\end{matrix}\right]=U^2L^{-2}U^2L^{-2}$, it is left to notice that $A_1$
represents a simple cycle of length 4, as is shown in Figure \ref{cycles-fig}, and therefore, $\lambda$ can be contracted to a multiple of this cycle.
\end{proof}

\subsection{Resolution of the ambiguity due to the cycle formed by $4$ vertices $(\pm1,\pm1)$}\label{ambiguity}
The presence of the cycle on the fifth component from Proposition \ref{trees}
(see the leftmost sketch of Figure \ref{cycles-fig}) requires some care to avoid multiple forms of continued fractions in Proposition \ref{spin-frac}
and the corresponding matrix factorizations in Corollary  \ref{cor-matrices}.
We resolve this issue by imposing condition $q_r\ne-2$ in Proposition \ref{spin-frac}.

\begin{Example}\label{ex1}
For  $(m,n)=(3,1)$
our method gives a continued fraction $\frac31=2+\frac11=[2,1]$, but forbid a fraction $\frac31=4+\frac1{-2+\frac11}=[4,-2,1]$.
In terms of matrices, this means allowing factorization
$\left[\begin{matrix} 3 \\ 1\end{matrix}\right]=U^2\left[\begin{matrix} 1 \\ 1\end{matrix}\right]$,
while forbidding $\left[\begin{matrix} 3 \\ 1\end{matrix}\right]=U^4L^{-2}\left[\begin{matrix} -1 \\ -1\end{matrix}\right]$.
In the leftmost part of Figure \ref{cycles-fig}, this corresponds to allowing a one-step passing from $(3,1)$-vertex to $(1,1)$-vertex, but
forbidding a $3$-step passing from $(3,1)$ to $(-1,-1)$.
\end{Example}

\begin{Example}\label{ex2}
For $(m,n)=(3,-1)$ we allow $\frac{3}{-1}=-4+\frac11$, but forbid $\frac{3}{-1}=-2+\frac{1}{-2+\frac11}$, which can be interpreted in terms of matrices and
paths in Figure \ref{cycles-fig}, in a way similar to the previous example.
\end{Example}

The only indeterminacy remains for a presentation in the case $mn=-1$.

\begin{Corollary}\label{minimality-cor}
The presentation of $\frac{m}n$ by a continued fraction as in Proposition \ref{spin-frac} is unique, except for the case $mn=-1$.
\end{Corollary}

\begin{proof}
By Corollary \ref{cor-matrices}, we can pass to the language of factorizations into products of $L^{\pm2}$ and $U^{\pm2}$ and reformulate the uniqueness
of the continued fractions as the uniqueness of the corresponding path in the Cayley graph. The latter uniqueness is automatic in the four tree-component of Proposition \ref{trees},
since the form of a matrix factorization in Corollary \ref{cor-matrices} excludes a possibility of multiple passing the same edge in a path.

In the case of the non-tree component, with the 4-cycle formed by vertices $(\pm1,\pm1)$,  we excluded from the statement the initial vertices $(1,-1)$ and $(-1,1)$, 
for which a required path is not unique. With the other initial vertices $(m,n)$ our path will be unique because we 
avoid extra steps like in Example \ref{ex1} and do not switch between edges labeled $U^2$ and $L^2$ after coming to the
vertex $(1,-1)$ or $(-1,1)$, as explained in Example \ref{ex2}.
\end{proof}

\section{Comparison of the algorithms}\label{comparison-sec}

\subsection{Lifting of our untwisting algorithm to the torus coordinates}

Note that Euclid's algorithm splits into steps $(m,n)\mapsto (m\pm2n,n)$, and $(m,n)\mapsto(m,n\pm2m)$, which can be viewed as
moving along edges of $\G_{\til\P}$. So, this algorithm can be viewed as a path leading from a chosen vertex $(m,n)\subset\til \P$ to $\d$
in $\G_{\til\P}$. Similarly, an untwisting algorithm in Section \ref{algorithm} can be viewed as a path from $(a,b)$ to $\D$ in $\G_{\CC}$.

\begin{Proposition}\label{2-algorithms}
$(1)$ The projection $\til \P\to\CC(\M)$ gives a double covering of the Cayley graphs $\G_{\til\P}\to\G_\CC$.

$(2)$ A path in $\G_{\til\P}$ representing 
the Euclidean algorithm  for finding an even continued fraction of $\frac{p}q$ (see Proposition \ref{spin-frac})
covers the path in $\G_\CC$ representing the untwisting algorithm of Section \ref{algorithm} for a curve  $\g\subset\M$ with
Dynnikov coordinates $(a,b)$, whose covering curve $\til\g$ has torus coordinates $(p,q)$ $($see Theorem \ref{transform-thm}
for the correspondence $(p,q)\mapsto(a,b))$.
\end{Proposition}

\begin{proof}
The projection $\til \P\to\CC(\M)$ is two-to-one (see Lemma \ref{2-1}),
with the action of $\bar\G(2)$ in $\til\P$ covering the action of $\PMcg(\M)$ on $\CC(\M)$.
The isomorphism $\PMcg(\M)=\bar\G(2)$ identifies generators $U^2$, $L^2$ with $t_c$ and $t_d^{-1}$
(see Sect. \ref{modular-subsect} and Lemma \ref{core-arcs}), which gives a double covering, in which a path from $(p,q)$
representing Euclid's algorithm covers the path representing the untwisting algorithm for the corresponding $(a,b)$.
%
\end{proof}

\begin{Corollary}\label{component-correspondence}
The Cayley graph covering $\G_{\til\P}\to \G_\CC$ from Proposition \ref{2-algorithms} sends 
the 5 components of $\G_{\til\P}$ to the 3 components of $\G_\CC$ as follows:
\begin{itemize}\item
two components-orbits
from Corollary \ref{5orbits}(1) to the component of $c\in\D\subset\CC$,
\item
two components-orbits from Corollary \ref{5orbits}(2) to the component of $d\in\D\subset\CC$,
\item
the component-orbit of $\G_{\til\P}$
from Corollary \ref{5orbits}(3) to the component of $e\in\D\subset\CC$.
\end{itemize}
\end{Corollary}

\subsection{Proof of Theorem \ref{minimal-Dynnikov}}
Proposition \ref{2-algorithms} shows that it is sufficient to verify the efficiency of the covering Euclid's algorithm, namely,
that the number of steps, $\frac12(|q_0|+\dots+|q_r|)$, used to reach set $\d\subset\til \P$ from $(m,n)\in\til \P$ is minimal. 
In the case of even $mn$ the minimality is trivial, because the component of $\G_{\til\P}$ containing $(m,n)$ is a tree (see Proposition \ref{trees}),
in which Euclid's algorithm does not allow passing the same edge twice.
 In the case of odd $mn$, Proposition \ref{trees} implies that the only possible ambiguity in the choice of a path from $(m,n)$ to $\d$
 is related to the 4-cycle formed by the vertices $(\pm1,\pm1)$. This issue is resolved 
in Section \ref{ambiguity}, where we observed that reaching vertex $(\e,\e)$, $\e\in\{1,-1\}$,
terminates Euclid's algorithm, while reaching $(\e,-\e)$ takes only one additional step.
%
%
\qed

\section{Proofs of the main theorems}\label{main-proofs}

\subsection{Proof of Theorem \ref{conjugacy}}
Existence of precisely 3 orbits indicated in the theorem follows from Proposition \ref{3-orbits}.
These orbits are related to the covering 5 orbits of the $\til\G(2)$-action on $\til \P$ via Proposition \ref{2-algorithms}(Part 1).
The latter 5 orbits are described in Corollary \ref{5orbits} in terms of parity of the components $p,q$ of the torus coordinates $(p,q)$
of an essential curve in $T$. So it remains to use the inversion formulas of Proposition \ref{reversion} to deduce the description of the 3 orbits
in terms of Dynnikov $(a,b)$-coordinates from the description in terms of $(p,q)$-coordinates.
Namely, by the inversion formulas, $b$ is even if and only if $p,q$ are both odd (recall that $p$ and $q$ are coprime),
and the $(p,q)$-orbit in $\til\P$ with odd $p,q$ covers the orbit of $t_e$ in $\CC(\DD_3)$.
Consider now the cases with odd $b$, or equivalently, $p$ and $q$ of different parity.
Then, by the inversion formulas, $p$ is odd if either
$b>0$ and $a$ is even, or $b<0$ and $a$ is odd. These cases represent the orbit of $c$ in $\CC(\DD_3)$, see Corollary \ref{component-correspondence}.
The remaining cases with even $p$ similarly represent the orbit of $d$.

\subsection{Proof of Theorem \ref{length}}
By Theorem \ref{minimal-Dynnikov}, the untwisting algorithm of Section \ref{algorithm} is of minimal length, if we consider it as a path in the graph $\G_\CC$ connecting $\g\in\CC(\DD_3)$ 
to the set $\D$.
This path, by Proposition \ref{2-algorithms}, is covered by a path in $\G_{\til\P}$ of length $|\frac{p}q|_{ECF}=\frac12(|q_0|+\dots+|q_r|)$, 
where $(p,q)\in\til\P$ covers $\g$ and $\frac{p}q$ is presented, like in Proposition \ref{spin-frac}, either as
$[q_0,\dots,q_r]$ or as $[q_0,\dots,q_r,1]$.
Such a presentation requires application of $|\frac{p}q|_{ECF}$ elementary steps of
the Euclidean algorithm (such steps can be interpreted as matrix multiplications by $L^{\pm2}$ or $U^{\pm2}$
as in Corollary \ref{cor-matrices}).

Finally, it remains to notice that
 the components $p,q$ are expressed by
Proposition \ref{reversion}  in terms of the Dynnikov coordinates $(a,b)$ of $\g$ in the same way as indicated in Theorem \ref{length}
(except for a possible simultaneous change of sign, which does not affect to $\frac{p}q$).


\subsection*{Acknowledgements}
We thank M.Korkmaz for a useful remark on the conjugacy problem.
We also thank the Max Planck Institute for Mathematics in Bonn for its hospitality, excellent working conditions, and support in summer 2024, when this work was initiated.


\bigskip
\providecommand{\bysame}{\leavevmode\hboxto3em{\hrulefill}\thinspace}


\begin{thebibliography}{1}

\bibitem{ADMY} E. Dalyan, E. Medeto\~{g}ullari, S.\"{O}. Yurtta\c{s} and F. Atalan, \textit{Generating the free group of rank two with Dynnikov Coordinates}, Cumhuriyet Sci. J., 47(2) (2026) 356--360.

\bibitem{D} I.A. Dynnikov, \textit{On a Yang-Baxter mapping and the Dehornoy ordering}, Russian Mathematical  Surveys,  57(3) (2002) 592--594.

\bibitem{FM} B. Farb and D. Margalit, \textit{A primer on mapping class groups},
Princeton University Press, New Jersey, 2012.

\bibitem{FLP} A. Fathi, F.Laudenbach and V. Poenaru, \textit{Travaux de Thurston sur les surfaces}, Seminaire Orsay, Asterisque, 66--67, 1979.

\bibitem{HY} T. Hall and S.\"{O}. Yurtta\c{s}, \textit{On the topological entropy of families of braids}, Topology and its Applications, 156 (2009) 1554–1564.

\bibitem{L} C. L\"{o}h, \textit{Geometric Group Theory - An introduction}, Springer, Switzerland, 2017.

\bibitem{Sch}  F. Schweiger,
\textit{Continued fractions with odd and even partial quotients}, Arbeitsberichte Math. Institut
Universit¨at Salzburg, 4 (1982), 59–70.


\bibitem{Th} J.-L. Thiffeault, \textit{Braids and Dynamics}, Springer, 2022.

\end{thebibliography}
\end{document}